\documentclass[11pt]{amsart}

 \usepackage{amsmath,amsthm,amsfonts,amssymb,verbatim}
 \usepackage{url}
 \usepackage{graphicx}
 \usepackage[all]{xy}
 \usepackage{wrapfig}
 \usepackage{picins}
 \usepackage{pinlabel}
 \usepackage{subfigure}
 
\newcommand\ons{Ozsv{\'a}th and Szab{\'o}}

\newcommand{\bZ}{\mathbb{Z}}
\newcommand{\bQ}{\mathbb{Q}}
\newcommand{\bR}{\mathbb{R}}

\newcommand{\pq}{\frac{p}{q}}
\newcommand{\pqzero}{\frac{p_0}{q_0}}
\newcommand{\pqone}{\frac{p_1}{q_1}}

\newcommand{\HFhat}{\widehat{\operatorname{HF}}}
\newcommand{\rk}{\operatorname{rk}}

\newtheorem{theorem}{Theorem}

\newtheorem{definition}[theorem]{Definition}
\newtheorem{question}[theorem]{Question}
\newtheorem{corollary}[theorem]{Corollary}
\newtheorem{proposition}[theorem]{Proposition}

\newtheorem{lemma}[theorem]{Lemma}

\title[Left-orderability and Dehn surgery]{Left-orderable fundamental groups and Dehn surgery}
\date{November 10, 2010}

\author[Adam Clay]{Adam Clay}
\address{
CIRGET, 
Universit\'e du Qu\'ebec \`a Montr\'eal, 
Case postale 8888, Succursale centre-ville, 
Montr\'eal QC, H3C 3P8.}
\email{aclay@cirget.ca}
\urladdr{http://thales.math.uqam.ca/~aclay}

\author[Liam Watson]{Liam Watson}
\thanks{Both authors partially supported by NSERC postdoctoral fellowships}
\address{Department of Mathematics, UCLA, 520 Portola Plaza, Los Angeles CA, 90095.}
\email{lwatson@math.ucla.edu}
\urladdr{http://www.math.ucla.edu/~lwatson}

\begin{document}

\begin{abstract}
There are various results that frame left-orderability of a group as a geometric property. Indeed, the fundamental group of a 3-manifold is left-orderable whenever the first Betti number is positive; in the case that the first Betti number is zero this property is closely tied to the existence of certain nice foliations. As a result, many large classes of 3-manifolds, including knot complements, are known to have left-orderable fundamental group. However, though the complement of a knot has left-orderable fundamental group, the result of Dehn surgery is a closed 3-manifold that need not have this property. We take this as motivation for the study of left-orderability in the context of Dehn surgery, and establish a condition on peripheral elements that must hold whenever a given Dehn surgery yields a manifold with left-orderable fundamental group. This leads to a workable criterion used to determine when sufficiently positive Dehn surgery produces manifolds with non-left-orderable fundamental group. As examples we produce infinite families of hyperbolic knots -- subsuming the $(-2,3,q)$-pretzel knots -- for which sufficiently positive surgery always produces a manifold with non-left-orderable fundamental group. Our examples are consistent with the observation that many (indeed, all known) examples of L-spaces have non-left-orderable fundamental group, as the given families of knots are hyperbolic L-space knots. Moreover, the behaviour of the examples studied here is consistent with the property that sufficiently positive surgery on an L-space knot always yields an L-space. 
\end{abstract}

\maketitle

\section{Introduction}

This paper studies left-orderable groups in the context of 3-manifold topology, and in particular, the question of whether there exist left-orderings of a fundamental group that descend to certain quotients of the group. To begin, we recall the following:

\begin{definition} A non-trivial group $G$ is left-orderable if there exists a strict total ordering $>$ of the elements of $G$ that is left-invariant: Whenever $g>h$ then $fg>fh$, for any $g,h,f\in G$. \end{definition}

We adhere to the somewhat non-standard convention that the trivial group is not left-orderable. An equivalent definition of left-orderability may be given in terms of positive cones; our convention is consistent with requiring that the positive cone be non-empty.

Left-orderability may be viewed as a geometric property: A countable group is left-orderable if and only if it acts faithfully on $\mathbb{R}$ by order-preserving homeomorphisms. In fact, this leads to a connection between left-orderability and $3$-manifolds as Boyer, Rolfsen and Wiest show that compact, connected, orientable manifolds supporting $\mathbb{R}$-covered foliations\footnote{We suppose throughout that all foliations are co-orientable. For the relevant definitions of $\bR$-covered and taut foliation see \cite{Calegari2007}, for example.} have left-orderable fundamental group. For Seifert fibred manifolds the connection is even stronger: a Seifert fibred space with base orbifold $S^2$ has left-orderable fundamental group if and only if it supports a taut foliation  \cite{BRW2005}. The work of Calegari and Dunfield \cite{CD2003} and Roberts, Shareshian and Stein \cite{RSS2003} provides related results in the hyperbolic setting. In particular, Calegari and Dunfield show that atoroidal rational homology spheres admitting a taut foliation have virtually left-orderable fundamental group (that is, there exists a left-orderable subgroup of finite index). Moreover, restricting to integer homology spheres, the existence of a taut foliation implies left-orderable fundamental group, so that in this setting establishing non-left-orderability of the fundamental group can be a useful obstruction to such a foliation of the manifold \cite{CD2003}. 


Burns and Hale \cite{BH1972} characterize left-orderability in terms of finitely generated subgroups. 

\begin{theorem}[\bf Burns-Hale] A group $G$ is left-orderable if and only if every nontrivial finitely generated subgroup surjects onto a left-orderable group. \end{theorem} 

While using this characterization on an arbitrary group is often an intractable problem,  Boyer, Rolfsen and Wiest combine this result with Scott's Compact Core Theorem \cite{Scott1973} to give a natural characterization of left-orderability for fundamental groups of 3-manifolds \cite{BRW2005}.

\begin{theorem}[\bf Boyer-Rolfsen-Wiest]
\label{thm:BRWcriterion}
Let $M$ be a compact, connected, irreducible, orientable 3-manifold, possibly with boundary. Then $\pi_1(M)$ is left-orderable if and only if  $\pi_1(M)$ surjects onto a left-orderable group.\end{theorem}

It follows immediately that whenever $H_1(M;\bQ)\ne0$ the group $\pi_1(M)$ is left-orderable. On the other hand, this characterization can be quite subtle for rational homology spheres. In particular, while it follows from this result that the complement of any knot has left-orderable fundamental group, this need not be the case for a rational homology sphere resulting from Dehn surgery on the knot (an operation described below). As a result, the following question seems well motivated and is the principal focus of this work: 

\begin{question}\label{qst:LOsurgery} How does left-orderability of fundamental groups behave with respect to Dehn surgery?\end{question}

To this end, suppose that $M$ is a compact, connected, irreducible, orientable 3-manifold, with $\partial M\cong S^1\times S^1$.  We recall the notion of Dehn filling, following the conventions of Boyer \cite{Boyer2002}.  Given a primitive class $\alpha\in H_1(\partial M;\bZ)/\pm1$ (referred to as a slope) we obtain a closed 3-manifold $M(\alpha)$ via Dehn filling by identifying the boundary of a solid torus $D^2\times S^1$ to the boundary of $M$ in such a way that $\partial D^2\times \{\operatorname{point}\}$ is identified with $\alpha$.   In the case that $M$ is the exterior of a knot $K$ in $S^3$, there is a canonical basis for the fundamental group of the boundary (the peripheral subgroup) given by the knot meridian $\mu$ and the Seifert longitude $\lambda$.  Thus, every slope $\alpha$ may be written in the form $\alpha = \pm(p\mu+q\lambda)$ where the elements $\{ \mu,\lambda \}$ generate the group $\pi_1(\partial M) \cong H_1(\partial M;\bZ)$. In this setting, we generally denote $M(\alpha)$ by $S^3_{p/q}(K)$ and refer to this manifold as the result of Dehn surgery (also referred to as $\pq$-surgery) along the knot $K$.  On the level of the fundamental group Dehn filling corresponds to killing the surgery slope $\alpha$, in the sense that $\pi_1(M(\alpha))\cong \pi_1(M)/\left<\left<\alpha\right>\right>$.

In this setting, Question \ref{qst:LOsurgery} asks if there exists a left-ordering of $\pi_1(M)$ that descends to a left-ordering of the quotient $\pi_1(M(\alpha))$.  For example, such left-orderings exist when $M$ is the complement of the figure eight knot and $\alpha$ is a slope in the interval $(-4, 4)$  \cite{BGW2010}, or when $M$ is the complement of a positive $(r,s)$ torus knot and $\alpha$ is a slope smaller than $rs-r-s$ \cite{BGW2010,Peters2009}.  The tools used in each of these examples are very specific to the geometry of the given situation, and it is not clear how either method might generalize to accommodate a larger class of knots.  Moreover, while the question of preserving (or not preserving) left-orderability under quotients has been studied in the field of orderable groups, few of the general theorems available are applicable to the specialized set-up of Dehn surgery  (see \cite[Chapter 3 and Chapter 5]{KM1996} for a summary of some available results).

Our main result introduces a workable obstruction to the existence of a left-ordering of $\pi_1(M)$ that descends to a left-ordering of $\pi_1(M(\alpha))$. That is, we provide a condition on the left-orderings of the group  $\pi_1(M)$ which ensures that the quotient in question will not inherit any left-ordering. 

\begin{theorem}[\bf see Corollary \ref{cor:interval}] 
Let $M$ denote the exterior of a non-trivial knot $K$ in $S^3$, and fix the canonical basis $\{ \mu,\lambda \}$ for the peripheral subgroup $\pi_1(\partial M)$, as above. If the implication 
\[\mu^{p_0}\lambda^{q_0}>1  \Rightarrow \mu^{p_1}\lambda^{q_1}>1
\]
holds for every left-ordering $>$ of $\pi_1(M)$ then $\pi_1(S^3_{p/q}(K))$ is not left-orderable for any $\frac{p}{q}\in ( \frac{p_0}{q_0},\frac{p_1}{q_1} )$.
\end{theorem}

While this criterion applies to a more general class of manifolds (see in particular Theorem \ref{thm:main-criteria}), we are able to apply a variant of this criterion to certain classes of knots in $S^3$. In particular, we show that sufficiently positive surgery on positive $(r,s)$-torus knots gives rise to manifolds with fundamental group that cannot be left-ordered (see Theorem \ref{thm:torus}). Since the result of such a surgery is Seifert fibred \cite{Moser1971}, this recovers a known result (see discussion below) via different means. However, we are also able to extend our calculation to cover a large class of hyperbolic knots. We prove:

\begin{theorem}[\bf see Theorem \ref{thm:pretzel} and Theorem \ref{thm:twisted}]\label{thm:twisted-intro} Let $K$ be a positively $m$-twisted $(3,3k+2)$-torus knot.
If (1) $m\ge0$ and $k=1$, or (2) $m=1$ and $k\ge0$, then $\pi_1(S^3_r(K))$ is not left-orderable for $r\in\bQ$ sufficiently positive.\end{theorem} 

The notion of a positively twisted torus knot is introduced in Section \ref{sec:torus}. This class includes, for example, the $(-2,3,q)$-pretzel knots for odd $q>5$ (note that the cases $q=1,3,5$ correspond to torus knots). As an immediate consequence, the result of sufficiently positive surgery on any of the knots considered in this theorem is a hyperbolic $3$-manifold that does not admit an $\bR$-covered foliation. Indeed, the criterion applied to establish this fact may be a useful algebraic tool for obstructing such a foliation in other contexts. However, as an $\bR$-covered foliation is an instance of a taut foliation, it is interesting to note that this fact may be obtained by other means for the knots in question. We describe this further in the next section. 

\subsection{Background and further motivation} While the main focus of this work is the interplay between left-orderability and Dehn surgery, part of our motivation for studying left-orderability of fundamental groups in this context comes from Heegaard Floer homology. Let $Y$ be a closed, connected, oriented 3-manifold, and denote by $\HFhat(Y)$ the (`hat' version of) Heegaard Floer homology of $Y$ \cite{OSz2004-invariants,OSz2004-properties}. We will be particularly interested in a class of manifolds introduced in \cite{OSz2005-lens} for which the Heegaard Floer homology is a simple as possible.  

\begin{definition} A closed, connected, orientable 3-manifold $Y$ is an L-space if it is a rational homology sphere satisfying $\rk\HFhat(Y)=|H_1(Y;\bZ)|$.  \end{definition}

Examples of L-spaces include lens spaces, as well as manifolds admitting elliptic geometry \cite[Proposition 2.3]{OSz2005-lens}.

An interesting topological property enjoyed by this class of manifolds is that they do not admit taut foliations, according to a result of \ons\ \cite[Theorem 1.4]{OSz2004-genus}. Restricting to the class of Seifert fibred spaces with base orbifold $S^2(a_1,\ldots,a_n)$, a converse to this result has been established by Lisca and Stipsicz \cite[Theorem 1.1]{LS2007}. That is, non-L-spaces within this class are characterized by the existence of a taut foliation. Combined with a result of Boyer, Rolfsen and Wiest \cite[Theorem 1.3(b)]{BRW2005}, this characterization may be restated in terms of left-orderability of the fundamental group, as observed by Peters \cite{Peters2009} and Boyer, Gordon and Watson \cite{BGW2010}. In fact, with this observation as a point of departure, it has been established that a Seifert fibred space (without restriction on the base orbifold) is an L-space if and only if its fundamental group cannot be left-ordered \cite{BGW2010} (see also \cite{Watson2009}).

This suggests a correspondence between L-spaces and non-left-orderability of fundamental groups, a phenomenon that has been studied further in various settings \cite{BGW2010,CR2010,Peters2009}. In light of this correspondence, it is natural to ask whether properties enjoyed by L-spaces may be translated into statements regarding non-left-orderability of the fundamental group.  

With respect to Dehn filling, L-spaces obey the following property (see \cite[Proposition 2.1]{OSz2005-lens}). Fix a compact, connected, orientable 3-manifold $M$ with torus boundary. Given a pair of slopes $\alpha$ and $\beta$ in $\partial M$ with geometric intersection 1 and $|H_1(M(\alpha);\bZ)|+|H_1(M(\beta);\bZ)|=|H_1(M(\alpha+\beta);\bZ)|$, if $M(\alpha)$ and $M(\beta)$ are L-spaces then so is $M(\alpha+\beta)$.  In the setting of surgery on a knot in $S^3$, this property may be restated a follows: If $S^3_n(K)$ is an L-space for some integer $n>0$, then $S^3_r(K)$ is also an L-space for any rational number $r\ge n$. A knot $K$ admitting $L$-space surgeries will be referred to as an L-space knot; examples are provided by torus knots and, more generally, Berge knots -- those knots known to admit lens space surgeries \cite{Berge}. Note that, up to taking mirrors, we may always assume that the integer $n$ is positive. We will restrict ourselves to considering positive surgeries in this work.  

Given this property of $L$-spaces with respect to Dehn surgery, the correspondence between L-spaces and the non-left-orderability of fundamental groups suggests that L-space knots should yield large families of 3-manifolds with fundamental group that cannot be left-ordered. In particular, one is led to ask:

\begin{question}\label{qst:large-surgery} Is the fundamental group of the manifold obtained by sufficiently positive surgery on an L-space knot non-left-orderable? \end{question}

More generally, one might consider a manifold with torus boundary (as above) for which neither $\pi_1(M(\alpha))$ nor $\pi_1(M(\beta))$ may be left-ordered, and ask whether $\pi_1(M(\alpha+\beta))$ is non-left-orderable. However, in the context of left-orderability questions pertaining to the relationship between three such quotients seem somewhat unnatural (or, at very least, unstudied). We leave this to future work. 

This paper -- in particular Corollary \ref{crl:main-criteria} and its applications -- constitutes, in part, an attempt to better understand 
Question \ref{qst:large-surgery}. Note that the infinite family of knots considered in Theorem \ref{thm:twisted-intro} are all L-space knots.

\subsection{Organization} The remainder of the paper is organized as follows. Section \ref{sec:criteria} contains the proof of our main criterion and its corollaries. In Section \ref{sec:torus} we introduce a family of twisted torus knots, compute their fundamental groups, and fix a choice of generators for the peripheral subgroups. These serve as a class of examples to which we may apply our obstructions to left-orderable fundamental groups after Dehn filling, the focus of Section \ref{sec:examples}.  

\subsection*{Acknowledgements} The authors thank Steve Boyer and Dale Rolfsen for their encouragement and input on this work, as well as Tye Lidman for helpful comments on an earlier version of the manuscript.

\section{A criterion for obstructing left-orderability}\label{sec:criteria}

The following criterion establishes a necessary condition on peripheral elements, given a left-orderable fundamental group arising from a Dehn filling. \begin{theorem} 
\label{thm:main-criteria}
Let $\frac{p}{q}, \frac{p_0}{q_0}, \frac{p_1}{q_1} \in \mathbb{Q}^+$ be given, with $ \frac{p}{q} \in( \pqzero, \pqone)$ and $p, q, p_i, q_i >0$.   Suppose that $M$ is compact, connected, orientable $3$-manifold with incompressible torus boundary, and suppose that $\langle\mu,\lambda\rangle\cong\pi_1(\partial M)$ is not sent to $1$ under the quotient map $\pi_1(M) \rightarrow \pi_1(M(\frac{p}{q}))$. If $M(\frac{p}{q})$ is left-orderable, then there exists a left-ordering of $\pi_1(M)$ relative to which the elements $\mu^{p_0} \lambda^{q_0}$ and  $\mu^{p_1} \lambda^{q_1}$ have opposite signs.
\end{theorem}

Before turning to the proof of this result, we record some immediate corollaries that will be used in the sequel. These come in the form of sufficient conditions to conclude that certain Dehn surgeries will give rise to non-left-orderable fundamental groups. In Corollaries \ref{cor:interval} and \ref{crl:main-criteria}, $M$ denotes the exterior of a nontrivial knot in $S^3$, with canonical generators $\mu$ and $\lambda$ for the peripheral subgroup. 

\begin{corollary}
\label{cor:interval}
Let $\frac{p}{q}, \pqzero, \pqone \in \mathbb{Q}^+$ be given, with $ \frac{p}{q} \in ( \pqzero, \pqone)$ and $p, q, p_i, q_i >0$. If $\mu^{p_0} \lambda^{q_0}>1$ implies  $\mu^{p_1} \lambda^{q_1}>1$ for every left-ordering $>$ of $\pi_1(M)$, then $\pi_1(M(\frac{p}{q}))$ is not left-orderable.
\end{corollary}
\begin{proof}
For contradiction, suppose that $\pi_1(M(\frac{p}{q}))$ is left-orderable.  Since $M$ is the complement of a nontrivial knot, $\pi_1(M(\frac{p}{q}))$ is nontrivial \cite{KM2004} (note that we could alternatively appeal to our convention that the trivial group is not left-orderable), and hence $\pi_1(\partial M)$ is not sent to $1$ under the quotient map $\pi_1(M) \rightarrow \pi_1(M(\frac{p}{q}))$.  Thus, we may apply Theorem \ref{thm:main-criteria} to conclude that there exists a left-ordering $>$ of $\pi_1(M)$ such that $\mu^{p_0} \lambda^{q_0}$ and  $\mu^{p_1} \lambda^{q_1}$ have opposite signs.  

If $\mu^{p_0} \lambda^{q_0}>1>\mu^{p_1} \lambda^{q_1}$, this contradicts the hypothesis that $\mu^{p_0} \lambda^{q_0}>1$ implies  $\mu^{p_1} \lambda^{q_1}>1$ for every left-ordering $>$ of $\pi_1(M)$.   If  $\mu^{p_1} \lambda^{q_1}>1>\mu^{p_0} \lambda^{q_0}$, then the opposite ordering of $\pi_1(M)$ contradicts our hypothesis.
\end{proof}

\begin{corollary}\label{crl:main-criteria}
Let $\frac{p}{q}, r\in \mathbb{Q}^+$ be given, with $p, q >0$ and $r > \frac{p}{q}$.  If $\mu^p \lambda^q>1$ implies $\mu^{p+N} \lambda^q>1$ for all $N>0$, then $\pi_1(M(r))$ is not left-orderable.
\end{corollary}
\begin{proof}
Choose $N>0$ such that $r \in ( \frac{p}{q}, \frac{p+N}{q})$, and apply Corollary \ref{cor:interval}.
\end{proof}

\subsection{The proof of Theorem \ref{thm:main-criteria}} We now collect the requisite material for the proof of our main result.

\begin{definition}
Suppose that $(G, >)$ is a left-ordered group.  A subgroup $C$ of $G$ is said to be convex relative to the left-ordering $>$ of $G$ if for every $f, h \in C$ and $g \in G$ the implication \[f>g>h \Rightarrow g \in C\] holds.\end{definition}

The following propositions and definitions are standard, see \cite[Propositions 2.1.1 -- 2.1.3]{KM1996} for proofs.   Together with the observations that follow, they establish the role of convex subgroups in the present context.

\begin{proposition}
\label{prop:convex}
Suppose that $C$ is a nontrivial subgroup of the left-ordered group $G$ with ordering $>$.  Then $C$ is convex relative to the ordering $>$ if and only if the prescription  
\[g>h \Rightarrow gC \succ hC\]
provides a well-defined left invariant ordering $\succ$ of the set of left cosets $G/C$.
\end{proposition}

\begin{definition}
The left-ordering $\succ$ of the set of cosets $G/C$ in Proposition \ref{prop:convex} is called the quotient ordering of $G/C$ arising from the left-ordering $>$ of $G$, and will be denoted by $>_{G/C}$.
\end{definition}

\begin{definition}Suppose that $(G, >)$ is a left-orderable group, and that $K$ is a nontrivial subgroup of $G$.  The restriction of $>$ to the subgroup $K$ will be denoted by $>_K$, and is called the restriction ordering of $K$ arising from $>$.
\end{definition}

\begin{proposition}
\label{prop:ses}
Suppose that 
\[1 \rightarrow K \stackrel{i}{\rightarrow} G \stackrel{\phi}{\rightarrow} H \rightarrow 1
\]
is a short exact sequence of nontrivial groups.  Then $K$ is convex relative to the left-ordering $>$ of $G$ if and only if $K$ and $H$ are left-orderable. Moreover, the left-ordering $>$ of $G$ is related to the left-orderings $>_K$ and $>_H$ by the following rule:  Given $g \in G$, if $\phi(g) \neq 1$, $g > 1$ if and only if $\phi(g) >_H 1$; otherwise, $\phi(g) =1$ and $g>1$ if and only if $g >_K1$.
\end{proposition}
\begin{proof}
The proof is routine.
\end{proof}

We remark that a version of this proposition holds even if $H$ does not have a group structure.  That is, we may use a left-invariant ordering of any convex subgroup $C \subset G$ and a left-invariant ordering of the set of left cosets $G/C$ in order to create a left-invariant ordering of $G$. 

\begin{lemma}
\label{lem:opp}
Let $\frac{p}{q}, \pqzero, \pqone \in \mathbb{Q}^+$ be given, with $p, q, p_i, q_i >0$, and suppose that $\{ (p,q), (u,v) \}$ is a basis for $\mathbb{R}^2$.  If $\frac{p}{q} \in ( \pqzero, \pqone)$ and 
\begin{eqnarray*}
 (p_0, q_0) & = & c_0(p,q) + d_0(u,v), \\
   (p_1, q_1) &=&  c_1(p,q)+d_1(u,v), 
 \end{eqnarray*}
then $d_0$ and $d_1$ have opposite signs.
\end{lemma}
\begin{proof}
Recall that the determinant $p_0q - q_0p$ is the signed area the parallelogram bounded by the vectors $(p_0, q_0)$ and $(p,q)$.  In particular, because the acute angle from $(p_0, q_0)$ to $(p,q)$ is swept out in a clockwise direction, the quantity $p_0q - q_0p$ is negative.  We rewrite:
\[p_0q - q_0p = (c_0p +d_0u)q - (c_0q + d_0v)p = d_0(uq-vp), 
\]
so the quantity $d_0(uq-vp)$ is negative.

On the other hand, the determinant $p_1q-q_1p$ is positive, and we compute that the quantity $d_1(uq-vp)$ is positive.  Therefore, $d_0$ and $d_1$ have opposite signs.
\end{proof}

\begin{proposition}
\label{prop:opp}
Let $\frac{p}{q}, \pqzero, \pqone \in \mathbb{Q}^+$ be given, with $p, q, p_i, q_i >0$.  Let $>$ be any ordering of $\mathbb{Z} \times \mathbb{Z}$ relative to which the subgroup $\langle (p, q) \rangle$ is convex.  If $\frac{p}{q} \in ( \pqzero, \pqone )$ then $(p_0, q_0)$ and $(p_1, q_1)$ have opposite signs in the ordering $>$ of $\mathbb{Z} \times \mathbb{Z}$. 
\end{proposition}
\begin{proof}

Consider the short exact sequence 
\[ 1 \rightarrow  \langle (p, q) \rangle \stackrel{i}{\rightarrow} \mathbb{Z} \times \mathbb{Z} \stackrel{\phi}{\rightarrow} \mathbb{Z} \rightarrow 1,
\]
where the map $\phi$ is the quotient map.  Since $\frac{p}{q} \in ( \pqzero, \pqone )$, neither $ \pqzero$ nor $ \pqone$ lies in the kernel of the map $\phi$.  Thus, by  Proposition \ref{prop:ses} we must show that $\phi(p_0, q_0)$ and $\phi(p_1, q_1)$ have opposite signs in order to show that $(p_0, q_0)$ and $(p_1, q_1)$ have opposite signs relative to the left-ordering $>$ of $\mathbb{Z} \times \mathbb{Z}$.

Choose $(u,v) \in \mathbb{Z} \times \mathbb{Z}$ such that $\phi(u,v) =1$.  Writing
 \begin{eqnarray*}
 (p_0, q_0) & = & c_0(p,q) + d_0(u,v), \\
   (p_1, q_1) &=&  c_1(p,q)+d_1(u,v),
 \end{eqnarray*}
we see that $\phi(p_0, q_0) = d_0$ and $\phi(p_1, q_1) = d_1$.  From Lemma \ref{lem:opp}, we know that $d_0$ and $d_1$ have opposite signs.
\end{proof}

\begin{lemma}
\label{lem:rescon}
Let $H$ be a nontrivial subgroup of the left-ordered group $(G, >)$.  If $C$ is convex relative to the left-ordering $>$, then $C \cap H$ is convex relative to the restriction ordering $>_H$.
\end{lemma}
\begin{proof}
The proof is routine.
\end{proof}

\begin{proposition}
\label{prop:res}
Suppose that $M$ is compact, connected, orientable $3$-manifold with incompressible torus boundary, and let $\alpha$ be a slope in  $\partial M$, and suppose that $\pi_1(\partial M)$ is not sent to $1$ under the quotient map $\pi_1(M) \rightarrow \pi_1(M(\alpha))$.  If $\pi_1(M(\alpha))$ is left-orderable, then we may define a left-ordering $>$ of $\pi_1(M)$ such that $\left< \alpha \right>$ is convex relative to the restriction ordering $>_{\pi_1(\partial M)}$ of $\pi_1(\partial M) \cong \mathbb{Z} \times \mathbb{Z}$.
\end{proposition}
\begin{proof}
Suppose that $\pi_1(M(\alpha))$ is left-orderable, and denote by $\left< \left< \alpha \right> \right>$ the normal closure of $\alpha $ in $\pi_1(M)$.  In particular, $\left< \left< \alpha \right> \right>$ is a left-orderable group, because it is a subgroup of $\pi_1(M)$, which is left-orderable by Theorem \ref{thm:BRWcriterion} (see \cite[Theorem 1.1]{BRW2005}).  By Proposition \ref{prop:ses}, we may use the short exact sequence
\[ 1 \rightarrow \left< \left< \alpha \right> \right> \rightarrow \pi_1(M) \rightarrow \pi_1(M(\alpha))\rightarrow 1
\]
to create a left-ordering $>$ of $\pi_1(M)$ relative to which $\left< \left< \alpha \right> \right>$ is a convex subgroup.

A nontrivial subgroup $C$ of $ \mathbb{Z} \times \mathbb{Z}$ is either isomorphic to a copy of the integers or it is of rank two.  In the case that $C$ is convex and of rank two, $C \cong \mathbb{Z} \times \mathbb{Z}$ is the whole group.  This follows from the observation that for convex groups, if $g^k\in C$ for some $k$ then $g\in C$. Thus, if $u$ and $v$ generate $C$ and $g$ is any other element of $ \mathbb{Z} \times \mathbb{Z}$, we may write $au+bv=kg$ for some integers $a,b$, hence $g$ is in $C$ and $C=\bZ\times\bZ$.   

We apply these observations to the subgroup $\left< \left< \alpha \right> \right> \cap \pi_1(\partial M)$, which is convex in $\pi_1(\partial M)$ relative to the left-ordering $<_{\pi_1(\partial M)}$ by Lemma \ref{lem:rescon}. In our present setting, since $\pi_1(\partial M)$ is not sent to $1$ under the quotient map $\pi_1(M) \rightarrow \pi_1(M(\alpha))$, we conclude that $\left< \left< \alpha \right> \right> \cap \pi_1(\partial M)$ must be isomorphic to a copy of the integers.  Since $\alpha$ is primitive, 
 \[\left< \left< \alpha \right> \right> \cap \pi_1(\partial M) = \left< \alpha \right>\]
and this subgroup is convex relative to the ordering $>_{\pi_1(\partial M)}$ of $ \pi_1(\partial M) \cong \mathbb{Z} \times \mathbb{Z}$.
\end{proof}

\begin{proof}[Proof of Theorem \ref{thm:main-criteria}]
With the above results in place, we are now in a position to prove our main criterion. Recall that $M$ is compact, connected, orientable $3$-manifold with incompressible torus boundary. By Proposition \ref{prop:res}, we may create a left-ordering $>$ of $\pi_1(M)$ such that $\left< \mu^p \lambda^q \right>$ is convex relative to the restriction ordering  $>_{\pi_1(\partial M)}$ of $\pi_1(\partial M)$.  By Lemma \ref{prop:opp}, the elements  $\mu^{p_0} \lambda^{q_0}$ and  $\mu^{p_1} \lambda^{q_1}$ have opposite signs in the restriction ordering  $>_{\pi_1(\partial M)}$, and hence they must have opposite signs in the ordering $>$ of $\pi_1(M)$. This completes the proof. 
\end{proof}

\subsection{A remark on bi-orderability}
Recall that a bi-ordering of a group is a left-ordering that is also invariant under right multiplication.  Work of Clay and Rolfsen has shown that the knot group of an L-space knot cannot be bi-ordered \cite{CR2010}. In light of Question \ref{qst:large-surgery} we have the following proposition, consistent with this observation. 

\begin{theorem}\label{thm:bi-order} If the group $\pi_1(S^3\smallsetminus\nu(K))$ is bi-orderable, then the hypothesis of Corollary \ref{cor:interval} (end hence that of Corollary \ref{crl:main-criteria} also) does not hold. That is, if $x, y \in \pi_1(S^3\smallsetminus\nu(K))$ are any two elements such that $\langle x \rangle \cap \langle y \rangle = \emptyset$, then there exists a left-ordering of $G$ such that $y>1>x$.\end{theorem} 

\begin{proof}
Recall that an element $g$ in a group $G$ is primitive if it cannot be written as $g=h^k$ for some $h \in G$ with $k>1$.  Every element in a knot group $\pi_1(S^3\smallsetminus\nu(K))$ may be written as a power of some primitive element \cite{Jaco1975}, and so we assume that $x = x_0^k$ where $x_0$ is primitive and $k>0$. 

Since $\pi_1(S^3\smallsetminus\nu(K))$ is bi-orderable and $x_0$ is primitive, there exists a left-ordering $>$ of $\pi_1(S^3\smallsetminus\nu(K))$ such that $x_0^{-1}$ is the smallest positive element \cite{ALR2008}.  If $>$ also satisfies $y>1$, then theorem is established.

If $1>y$, note that the subgroup $\langle x_0 \rangle$ is convex relative to the left-ordering $>$, and thus $>$ descends to a left-invariant ordering $\succ$ of the cosets $\pi_1(S^3\smallsetminus\nu(K))/\langle x_0 \rangle$.  Moreover, the coset $y\langle x_0 \rangle$ is different from $\langle x_0 \rangle$, since we are assuming $\langle x \rangle \cap \langle y \rangle = \emptyset$.  We may then define a left-ordering $>'$ of $\pi_1(S^3\smallsetminus\nu(K))$ relative to which $y$ is positive and $x$ is negative by reversing the ordering of the cosets, as in the following definition:

Given $g \in G$, if $g \notin \langle x_0 \rangle$, declare $g >' 1$ if and only if $1\succ g\langle x_0 \rangle $; otherwise, $g \in \langle x_0 \rangle$ and $g>1$ if and only if $g = x_0^k$ for some $k<0$ (c.f. Proposition \ref{prop:ses}).  It is easy to check that $>'$ so defined provides a left-ordering satisfying $y>'1>'x$, and the theorem follows.
\end{proof}

\section{Torus knots and related constructions}\label{sec:torus}

\subsection{Conventions for torus knots and their fundamental groups}

Denote the $(p,q)$-torus knot by $T_{p,q}$, where $p$ and $q$ are relatively prime, positive integers. The knot $T_{3,5}$, for example, is shown in Figure \ref{fig:torus-knot}. 

\begin{figure}[ht!]
\labellist \small
	\pinlabel $b$ at 194 428 
	\pinlabel $a$ at 282 400
	\endlabellist
\includegraphics[scale=1.0]{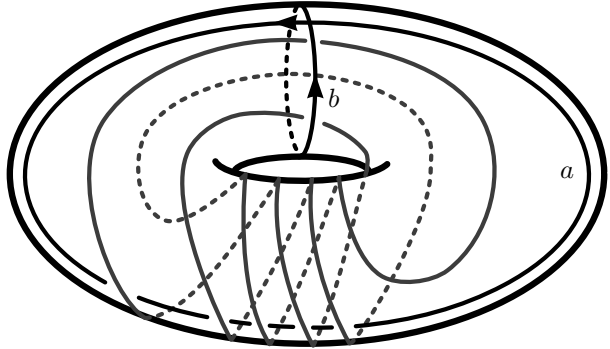}
\caption{The torus knot $T_{3,5}$, labelled with standard generators $a$ and $b$ for the knot group.}\label{fig:torus-knot}
\end{figure}

The knot group of the $T_{p,q}$ torus knot is given by the presentation
$G_{p,q} = \langle a, b | a^p = b^q \rangle$,
with generators $a$ and $b$ represented by the curves depicted in Figure \ref{fig:torus-knot}. We recall that one may arrive at this presentation immediately from an application of the Seifert-Van Kampen Theorem applied to the obvious genus 1 Heegaard decomposition of $S^3$, where $a$ and $b$ are generators the fundamental groups of the respective handlebodies. This point of view will be essential in what follows.  
  
Denote the meridian and longitude curves by $\mu$ and $\lambda$.  Relative to the generators $a, b$ we may write $\mu = b^j a^i$, where $i, j$ are integers satisfying $pj + qi = 1$, we may assume that $p>i>0$ and $0>j>-q$. Observe that by using the relation $a^p = b^q$, we may rewrite the meridian as $\mu = b^{q+j}a^{i-p}$. The longitude is given by $\lambda = \mu^{-pq} a^p=\mu^{-pq}b^q$, since $a^p$ (equivalently, $b^q$) specify the surface framing of the torus knot (see Moser \cite{Moser1971} for details).

These choices specify generators $b^ja^i$ and $a^p$ for the peripheral subgroup up to conjugacy.

\subsection{Preliminaries on twisted torus knots}\label{sub:twisted}  We begin by recording some facts about twisted torus knots. We will focus on the positively twisted $(3,q)$-torus knots, denoted $T_{3,q}^m$ where $m \geq 0$ denotes the number of (positive) full twists added along a pair of strands.  Further, we restrict our attention to the case when $q$ is congruent to $2$ modulo $3$ to streamline the discussion; the case when $q$ is congruent to 1 modulo 3 is similar.  This family of knots may be constructed by adding a second handle to the the standard splitting torus for the knot $T_{3,q}$, allowing two of the three strands to pass over the new handle, and finally adding $m$ positive full twists to the new handle. See Figure \ref{fig:torus-twisted} for the case $T_{3,5}^m$.

\begin{figure}[ht!]
\labellist \small
	\pinlabel $\cdots$ at 223 487	
	\pinlabel {$m$ full twists $\rightarrow$}  at 150 487
\endlabellist
\includegraphics[scale=1.0]{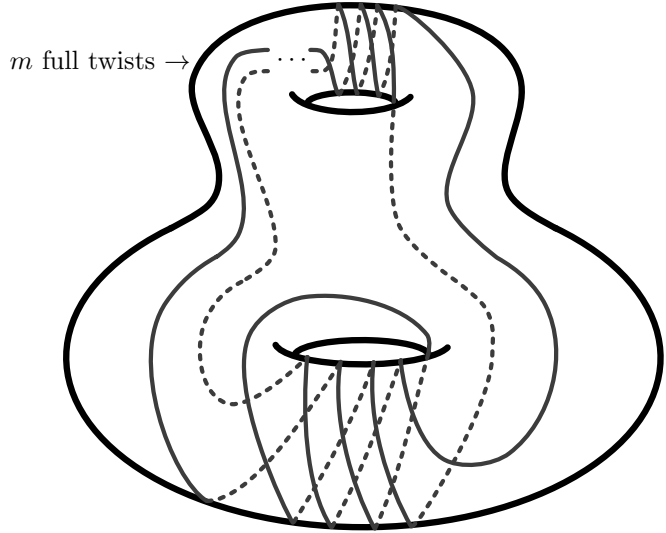}
\caption{The addition of a second handle to obtain the family of twisted torus knots $T^m_{3,5}$, where $m$ records the number of full twists.}\label{fig:torus-twisted}
\end{figure}

This is a simple construction giving rise to many familiar knot types. For example, it is an easy exercise to show that $T_{3,5}^m$ is the $(-2,3,5+2m)$-pretzel knot, denoted $P(-2,3,5+2m)$. This class pretzel knots provides an infinite family of hyperbolic L-space knots \cite{OSz2005-lens}, and includes, for example, the Berge knot $P(-2,3,7)$. 

\subsection{Fundamental groups of twisted torus knots}
We will use a genus two Heegaard decomposition of $S^3$ to compute the knot group of $T_{3,q}^m$,  which we will denote as $G^m_{3,q}=\pi^m_q$.  The computation follows the classical application of the Seifert-Van Kampen Theorem as in the case of the torus knots.  We restrict ourselves to the case $q=3k+2$ for $k\ge0$.

\begin{proposition}\label{prp:twisted-group-calculation}
Suppose that $m\ge0$ and $q = 3k+2$ for $k\ge0$.  Then the fundamental group of $S^3 \smallsetminus\nu(T_{3,q}^m)$ is
\[
\pi_q^m = \langle a,b |  a^2(b^{-k}a)^ma=b^{2k+1}(b^{-k}a)^mb^{k+1} \rangle. 
\] In particular, when $m=0$ we recover the torus knot group $G_{3,3k+2}$.
\end{proposition}

\begin{proof} We begin by fixing notation. Let $S^3=U\cup_\Sigma V$ be the genus two Heegaard splitting of $S^3$ specified in Figure \ref{fig:vankampen}, so that $\pi_1(\Sigma)$ is generated by $a,b,c,d$ and $\pi_1(U)$ and $\pi_1(V)$ are the free groups $\langle a,c \rangle$ and $\langle b,d \rangle$ respectively.   We will use the Seifert-Van Kampen Theorem to express the knot group $\pi_q^m$ as a free product of $\pi_1(U)$ and $\pi_1(V)$ with amalgamation.

\begin{figure}[ht!]
\labellist \small
	\pinlabel $b$ at 317 390
	\pinlabel $a$ at 222 390
	\pinlabel $d$ at 317 438
	\pinlabel $c$ at 222 438
\endlabellist
\includegraphics[scale=1.0]{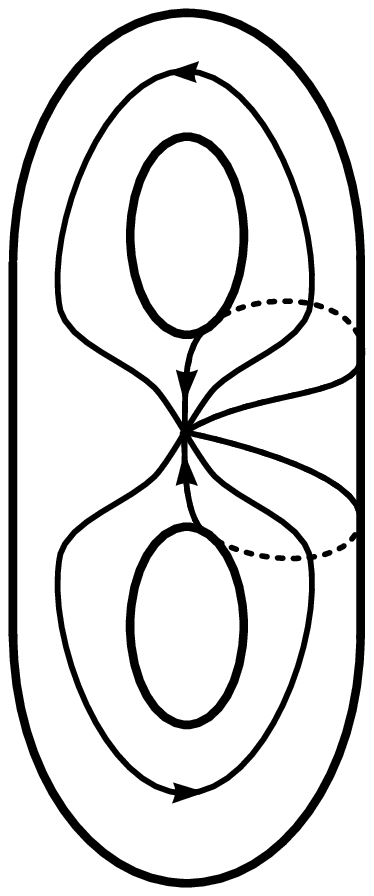}\quad\quad
\labellist \small
	\pinlabel \rotatebox{90}{$\underbrace{\phantom{aaaaa}}$} at 239 334
	\pinlabel $k$ at 248 334
	\pinlabel \rotatebox{-90}{$\underbrace{\phantom{aaaaa}}$} at 261 448
	\pinlabel $m$ at 250 448
\endlabellist
\includegraphics[scale=1.0]{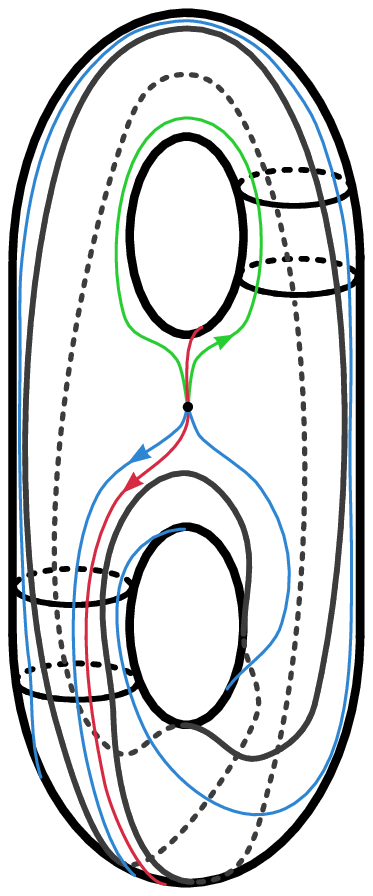}\quad\quad
\labellist \small
	\pinlabel \rotatebox{90}{$\underbrace{\phantom{aaaaa}}$} at 206 367
	\pinlabel $k$ at 215 367
	\pinlabel \rotatebox{-90}{$\underbrace{\phantom{aaaaa}}$} at 228 481
	\pinlabel $m$ at 217 481
	\pinlabel $\bar{\mu}$ at 207 399
	\pinlabel $\bar{s}$ at 193 419
\endlabellist
\includegraphics[scale=1.0]{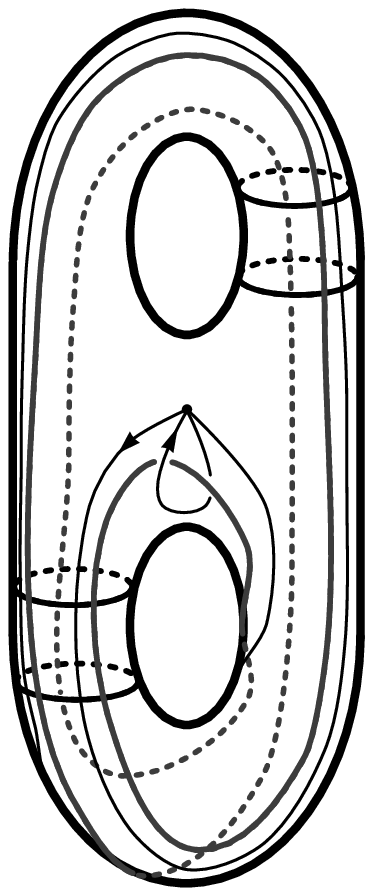}
\caption{Generators for the genus two splitting surface $\Sigma$ (left), a generating set for the fundamental group of $\Sigma\smallsetminus \nu( T^m_{3,q})$ (centre), and generators $\bar{\mu}$ and $\bar{s}$ for a peripheral subgroup (right). Note that the surface framing $\bar{s}$ is obtained by tracing the knot to give $\bar{s}=acaca$ where $c=(b^{-k}a)^m$.}\label{fig:vankampen}
\end{figure}

Therefore, we must first determine the image of the generators of $\pi_1(\Sigma\smallsetminus \nu (T^m_{3,q}))$ in each of the groups $\pi_1(U)$ and $\pi_1(V)$.  The generators in each case are represented by the oriented blue, green and red curves illustrated in Figure \ref{fig:vankampen}.

Consulting Figure \ref{fig:vankampen}, we see that the generator represented by the green curve may be written as $c$ in terms of the generators of $\pi_1(U)$, or as $d^m$ in terms of the generators of $\pi_1(V)$, and so we have the relation $c=d^m$.  From the red generator, we read off $a=(b^k)d$, and from the blue generator, $a^2ca=(b^k)b(b^k)d^m(b^k)b$. Since $d=b^{-k}a$ and $c=(b^{-k}a)^m$, we calculate \[\pi_q^m=\langle a,b |  a^2(b^{-k}a)^ma=b^{2k+1}(b^{-k}a)^mb^{k+1}\rangle \] where $q=3k+2$, so that $m$ and $k$ both record (positive) full twists on two and three strands, respectively. 
\end{proof}

\subsection{Determining the peripheral subgroup}  The peripheral subgroup of $\pi_q^m$ may be generated by the knot meridian and the surface framing of the knot, represented by those curves illustrated in Figure \ref{fig:vankampen}. Denote these elements by $\bar{\mu}$ and $\bar{s}$ respectively, in this section we will compute $\bar{\mu}$ and $\bar{s}$ in terms of the generators $a$ and $b$.  This done, we will fix a choice of peripheral subgroup generated by $\mu = a \bar{\mu} a^{-1}$ and $s = a \bar{s}a^{-1}$, as these generators are easier to work with in Section \ref{sec:examples}.

First notice that it is immediate from Figure \ref{fig:vankampen} that $\bar{s}=acaca$, where $c=(b^{-k}a)^m$ as in the previous section.

Next, we turn to determining $\bar{\mu}$ in terms of the generators $a$ and $b$.
To this end, we consider the portion of the knot that lies in the lower handle of the handlebody depicted in Figure \ref{fig:vankampen}.  This portion of the knot gives us a $3$-braid as in Figure \ref{fig:step1}, the generators $a$ and $b$ are as shown.   Recall that $\Delta^2$ is the full twist on three strands, which generates the cyclic centre of $B_3$.

\begin{figure}[ht!]\labellist \small
	\pinlabel $a$ at 171 455
	\pinlabel $b$ at 454 455
	\pinlabel $\Delta^{2k}$ at 222 415  
\endlabellist
\includegraphics[scale=0.75]{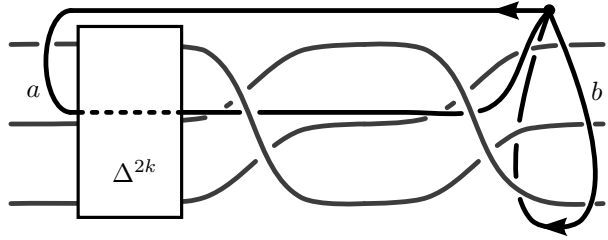}
\caption{Generators $a$ and $b$ for the knot group of the $(3,3k+2)$-torus knot. Recall that these generators correspond to one of the two handles corresponding to the Heegaard splitting for the knot $T^m_{3,3k+2}$.}\label{fig:step1}\end{figure}

 We first conjugate $\bar{\mu}$ by the generator $a$, this isomorphism by conjugation has the effect of moving $\bar{\mu}$ linking the top strand to $\mu$ linking the bottom strand of the associated 3-braid (see Figure \ref{fig:step2}).  

\begin{figure}[ht!]\labellist \small
	\pinlabel $\Delta^{2k}$ at 192 415
	\pinlabel $\mu=a\bar{\mu}a^{-1}$ at 195 375 
	\pinlabel $\bar{\mu}$ at 420 455 
\endlabellist
\includegraphics[scale=0.75]{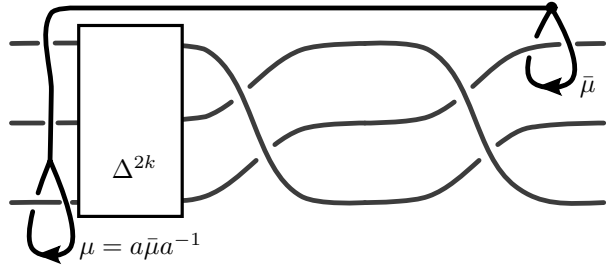}
\caption{Conjugation by $a$ relating $\mu$ and $\bar{\mu}$.}\label{fig:step2}\end{figure}

We proceed to deduce a formula describing $\mu = a \bar{\mu}a^{-1}$ by induction on the number of full twists, $k$.  Considering the case $k=0$, Figure \ref{fig:meridian} indicates a homotopy between $ba^{-1}$ and the meridian $\mu$.

\begin{figure}[ht!]\labellist \small
	\pinlabel $\mu=ba^{-1}$ at 255 377 
\endlabellist
\includegraphics[scale=0.75]{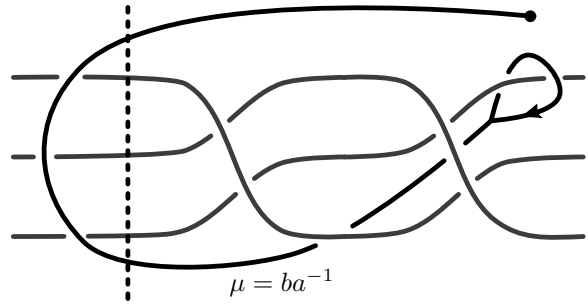}
\caption{Homotopy between the peripheral element $\mu$ and the word $ba^{-1}$ in the case $k=0$.}\label{fig:meridian}\end{figure}

With this case understood, we turn to adding full twists (as indexed by $k$ in Figure \ref{fig:vankampen}).  Adding a copy of the full twist $\Delta^2$ to our knot is accomplished by inserting the twist, appearing in Figure \ref{fig:full-twist}, at the dashed line in the base case illustrated in Figure \ref{fig:meridian}.  Note that the generator $b$, as it appears in Figure \ref{fig:full-twist}, is homotopic to the core of the handle.  This observation allows us to see that for each copy of $\Delta^2$ added to our knot, our meridian $\mu$ (linking the left bottom strand) is modified by prefixing the base case $\mu = ba^{-1}$ with a copy of the generator $b$.

\begin{figure}[ht!]\labellist \small
	\pinlabel $b$ at 185 391 
\endlabellist
\includegraphics[scale=0.75]{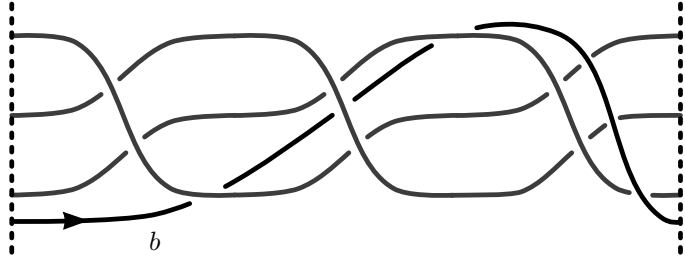}
\caption{The full-twist on 3 strands, denoted $\Delta^2$, together with the generator $b$, $k$ copies of which are inserted at the dashed line of Figure \ref{fig:meridian}. Notice that this curve is homotopic to the core of the handle.}\label{fig:full-twist}\end{figure}

Thus we arrive at the general formula $a \bar{\mu} a^{-1} = \mu = b^{k+1}a^{-1}$. By taking the conjugates $s = a \bar{s} a^{-1}$ and $\mu$, we arrive at:

\begin{proposition}
Suppose $q = 3k+2$, $k \geq 0$, and let $m\ge0$. 
We may take as generators of the peripheral subgroup of $\pi_q^m$ the elements 
$\mu = b^{k+1} a^{-1}$ and $s = a^2(b^{-k}a)^ma(b^{-k}a)^m$.
\end{proposition}

Note that this formula agrees with the formula for the meridian of the $(3, 3k+2)$-torus knot.  This is not a coincidence, as there exists a homotopy  (as shown in the above discussion) between $b^{k+1} a^{-1}$ and the loop $\mu$ in Figure \ref{fig:step2} which is supported away from the twists added to the torus knot $T_{3,q}$ to create the knot $T_{3,q}^m$.

For use in the next section, we record some identities that now follow immediately from the relation in $\pi^m_q$.  We have that
\begin{align*}\mu &= b^{k+1}a^{-1} = (b^{-k}a)^{-m}b^{-(2k+1)}a^2(b^{-k}a)^{m}\\ 
s&= a^2(b^{-k}a)^ma(b^{-k}a)^m=b^{2k+1}(b^{-k}a)^mb^{k+1}(b^{-k}a)^m
\end{align*}
and refer to these as the meridian and surface framing  respectively for the knot $T^m_{3,3k+2}$. Notice that as an immediate consequence of the construction, the canonical longitude is given by $\lambda=\mu^{-3(3k+2)-2m}s$. This is a direct generalization of the surface framing for the torus knot $T_{3,3k+2}$ in the case $m=0$.

\section{Examples and applications}
\label{sec:examples}

We now turn to applications of our criteria to produce infinite families of  rational homology spheres with non-left-orderable fundamental groups. With the obvious exception of surgery on torus knots, the constructions given here produce hyperbolic examples. All of the examples given are L-spaces.

\subsection{Surgery on torus knots} The work of Moser establishes that surgery on a torus knot always yields a Seifert fibred space, or a connect sum of lens spaces \cite{Moser1971}. As a result, these surgery manifolds are completely understood in terms of L-spaces and non-left-orderability (compare \cite{BGW2010,Peters2009}). However, our interest is in establishing that sufficiently large positive surgery on a torus knot gives rise to a manifold with non-left-orderable fundamental group by applying of Corollary \ref{crl:main-criteria}, directly from the presentation $G_{p,q} = \langle a, b | a^p = b^q \rangle$. 

\begin{theorem}\label{thm:torus} Let $K$ be a positive $(p,q)$-torus knot. Then $\pi_1(S^3_r(K))$ is not left-orderable whenever $r\ge pq-1$. \end{theorem}

This result follows immediately from Corollary \ref{cor:interval} and Corollary \ref{crl:main-criteria}, together with the following observations. 
 
\begin{lemma}
\label{lem:pq}
Suppose that $\mu^{pq} \lambda >1$ in some left-ordering of $G_{p,q}$. Then either
$b^{(q+j)j}a^{(q+j)i} >1$ or $ a^{(-j)(i-p)}b^{(-j)(q+j)} >1$.
\end{lemma}
\begin{proof}
Observe that $\mu^{pq} \lambda = a^p >1$, so $a>1$.  Assuming both $1> a^{(-j)(i-p)}b^{(-j)(q+j)}$ and $1>b^{(q+j)j}a^{(q+j)i} $, their product must also be negative, so we compute
\[1> a^{(-j)(i-p)}b^{(-j)(q+j)} b^{(q+j)j}a^{(q+j)i}= a^{pj+qi} = a,
\]
a contradiction.
\end{proof}

\begin{proposition}Suppose that $\mu^{pq} \lambda >1$ in some left-ordering of $G_{p,q}$. Then $\mu^{N+pq} \lambda >1$ for all $N >0$.
\end{proposition}
\begin{proof}
Assuming that $\mu^{pq} \lambda = a^p = b^q >1,$ we deduce that both $a$ and $b$ are positive elements as well.  We use this fact throughout the proof.
By Lemma \ref{lem:pq}, there are two cases to consider.

\noindent \textbf{Case 1:} $b^{(q+j)j}a^{(q+j)i} >1$. 

Suppose that $n>0$ is the smallest integer such that $b^{nj}a^{ni} >1$.  Then we may write:
\begin{eqnarray*}
 \mu^{N+pq} \lambda & = & (b^ja^i)^N a^p\\
   &=& b^{(1-n)j} (b^{nj}a^{ni}(a^{(1-n)i} b^{(1-n)j}))^N b^{(n-1)j} b^q. \end{eqnarray*}
As $n$ is the smallest integer for which $b^{nj}a^{ni} >1$, we have $1>b^{(n-1)j}a^{(n-1)i}$, and hence $a^{(1-n)i}b^{(1-n)j}>1$. 

Write $(1-n)j =kq +l$, where $0 \leq l < q$ and $k$ is positive.  This allows us to rewrite the above expression, and use the fact that $b^q$ is central, to arrive at:
\begin{eqnarray*}
 \mu^{N+pq} \lambda & = & b^{kq+l} (b^{nj}a^{ni}(a^{(1-n)i} b^{(1-n)j}))^N b^{-kq-l}b^q\\
   &=&b^{l} (b^{nj}a^{ni}(a^{(1-n)i} b^{(1-n)j}))^N b^{q-l}. 
\end{eqnarray*}
We conclude that $\mu^{N+pq} \lambda$ is positive, as it is a product of positive elements.

\noindent \textbf{Case 2:} $ a^{(-j)(i-p)} b^{(-j)(q+j)}>1$. 

As above, suppose that $n>0$ is the smallest integer such that $ a^{n(i-p)} b^{n(q+j)} >1$.  Then we may write:
\begin{eqnarray*}
 \mu^{N+pq} \lambda & = & (b^{q+j}a^{i-p})^N a^p\\
   &=& b^{n(q+j)}((b^{(1-n)(q+j)} a^{(1-n)(i-p)})a^{n(i-p)}b^{n(q+j)})^Nb^{q-n(q+j)}. 
\end{eqnarray*}
Since we have assumed $n$ minimal, both $ b^{(1-n)(q+j)}a^{(1-n)(i-p)}$ and $a^{n(i-p)}b^{n(q+j)}$ are positive.  Write $n(q+j) = kq+l$, where $0 \leq l<q$ and $k$ is positive.  Then we use the fact that $b^q$ is central to arrive at
\begin{eqnarray*}
 \mu^{N+pq} \lambda & = & b^{kq+l}((b^{(1-n)(q+j)} a^{(1-n)(i-p)})a^{n(i-p)}b^{n(q+j)})^Nb^{q-kq-l}. \\
& = & b^{l}((b^{(1-n)(q+j)} a^{(1-n)(i-p)})a^{n(i-p)}b^{n(q+j)})^Nb^{q-l}. 
\end{eqnarray*}
which is a product of positive terms as in Case 1, and the proof is complete.
\end{proof}

\begin{proposition}
Suppose that $\mu^{pq} \lambda >1$ in some left-ordering of $G_{p,q}$. Then $\mu^{pq-1} \lambda >1$.
\end{proposition}
\begin{proof}
If $\mu^{pq} \lambda = a^p = b^q >1$, then both $a$ and $b$ are positive.  Hence $\mu^{pq-1} \lambda  = (b^j a^i)^{-1} a^p = a^{p-i} b^{-j}$ where $p>i>0>j$. Therefore $\mu^{pq-1} \lambda$ is a product of positive elements, and so is positive.
\end{proof}

Lastly, we observe that the quotient $G_{p,q} / \langle \langle a^p \rangle \rangle$ is not left-orderable either, as both generators $a, b$ of $G_{p,q}$ are mapped to torsion elements in the quotient (in fact, Moser shows that the quotient is finite \cite{Moser1971}). Combining these observations, we have that surgery coefficients in the interval $(pq-1,\infty)=(pq-1,pq)\cup\{ pq \}\cup(pq,\infty)$ give rise to non-left-orderable fundamental groups, concluding the proof of Theorem \ref{thm:torus}. 

\subsection{Surgery on pretzel knots} We consider the L-space knots $T_{3,5}^m\simeq P(-2,3,5+2m)$ for $m\ge0$. The fundamental groups of these knots are given by 
\[\langle a, b | a^2(b^{-1} a)^m a = b^3(b^{-1}a)^m b^2 \rangle, 
\]
together with the peripheral elements 
\[ \mu = b^2 a^{-1} = (a^{-1}b)^{-m}b^{-3}a^2(a^{-1}b)^m
\]
and 
\[ s = \mu^{15+2m} \lambda = a^2 (b^{-1}a)^m a(b^{-1}a)^m = b^3(b^{-1}a)^mb^2(b^{-1}a)^m.  
\]

\begin{theorem}
\label{thm:pretzel}
If $r > 15+2m$ and $m\ge 0$, then $r$-surgery on the $(-2,3,5+2m)$-pretzel knot gives rise to a manifold with non-left-orderable fundamental group. 
\end{theorem}

By Corollary \ref{crl:main-criteria}, Theorem \ref{thm:pretzel} is an immediate consequence of the following.
\begin{proposition}
If $>$ is any left-ordering of $G$, then $s>1$ implies $\mu^N s >1 $ for all integers $N>0$.
\end{proposition}
\begin{proof}
Suppose that $s>1$, and let $N>0$ be any positive integer.  If $\mu>1$, then the conclusion $\mu^N s >1 $ follows trivially, as an arbitrary product of positive elements is always positive.  Assume then, without loss of generality, that $\mu = b^2 a^{-1}$ is negative (and $\mu^{-1}=a b^{-2} >1$).

We proceed by considering two cases, depending on whether the word $w=b^{-1}a$ is a positive or negative element. 

{\bf Case 1:} $1>w$.

Notice that in this case we have  \[s=a^2(b^{-1}a)^ma(b^{-1}a)^m=a(aw^m)^2\] so that $1>s$ as a product of negative elements, unless $a>1$. Since $1>w$ implies that $b>a$, and $s>1$ by assumption, we need only consider $b>a>1$. 

Note that in this setting $b^2>b>a$ implies that $a^{-1}b^2>1$ and write \[\mu^N=(b^2a^{-1})^N=b^2(a^{-1}b^2)^{N-1}a^{-1}\] and \[(b^{-1}a)^m=b^{-1}(ab^{-1})^{m-1}a=b^{-1}(\mu^{-1}b)^{m-1}a\]so that \[s=b(b^2(b^{-1}a)^m)^2 = b(b(\mu^{-1}b)^{m-1}a)^2 = b^2(\mu^{-1}b)^{m-1}ab(\mu^{-1}b)^{m-1}a.\] Now \[\mu^Ns=b^2(a^{-1}b^2)^{N-1}a^{-1}b^2(\mu^{-1}b)^{m-1}ab(\mu^{-1}b)^{m-1}a=b^2(a^{-1}b^2)^{N}(\mu^{-1}b)^{m-1}ab(\mu^{-1}b)^{m-1}a\] must be positive, as a product of the positive elements $a$, $b$, $\mu^{-1}$ and $a^{-1}b^2$.

{\bf Case 2:} $w>1$.

Writing \[(b^{-3}a^2)^N=b^{-2}((b^{-1}a)(ab^{-2}))^{N-1}b^{-1}a^2 = b^{-2}(w\mu^{-1})^{N-1}wa\] we have that \begin{align*} \mu^Ns=s\mu^N &= a(a(b^{-1}a)^m)^2(b^{-1}a)^{-m}(b^{-3}a^2)^N(b^{-1}a)^m \\ 
&= a^2w^mab^{-2}(w\mu^{-1})^{N-1}waw^m \\
&= a^2w^m(\mu^{-1}w)^{N}aw^m,
\end{align*}
immediately exhibiting $\mu^Ns>1$ as a product of positive elements whenever $a>1$. Since $w=b^{-1}a>1$ implies $a>b$, the assumption that $s>1$ reduces the proof to the case $1>a>b$.

Suppose then that $1>a>b$, and restrict to the case $m>0$ (noting that $m=0$ has already been handled in the torus knot setting). We claim that $a^2w^m>1$ and $aw^m>1$. To see this, notice that if $1>ab^{-1}$ then \[s=b(b^2(b^{-1}a)^m)^2=b(b(ab^{-1})^{m-1}a)^2\] is negative, as a product of negative elements. Since $s>1$ by assumption, it must be that $ab^{-1}>1$ and $1>ba^{-1}$. Now write \[(b^{-1}a)^m = b^{-2}((ba^{-1})(a^2b^{-2}))^{m-1}ba\] so that \[s=b(b^2(b^{-1}a)^m)^2=b(((ba^{-1})(a^2b^{-2}))^{m-1}ba)^2\] forcing $a^2b^{-2}>1$ and $1>b^2a^{-2}$ (otherwise $1>s$, a contradiction).

Now suppose that $1>a^2w^m$. Then \[1>(b^2a^{-2})(a^2w^m)=b^2w^m\] hence \[1>s=b(b^2w^m)^2\] as a product of negative elements, a contradiction. This proves that $a^2w^m>1$.

Similarly, $s=a(aw^m)^2>1$ forces $aw^m>1$. 

Now, revisiting the rewriting \[\mu^Ns=a^2w^m(\mu^{-1}w)^{N}aw^m,\] we see that $\mu^Ns>1$ is a product of the positive elements $\mu^{-1}$, $w$, $aw^m$ and $a^2w^m$. \end{proof}

\subsection{Surgery on torus knots with one added twist.} In this section we consider the family of twisted torus knots $T^1_{3,3k+2}$, where $m=1$ and $k\ge0$ varies.  These examples provide an infinite family of hyperbolic knots, as a consequence of Thurston's Hyperbolic Dehn Filling Theorem \cite{Thurston1980}. 

To see this, we observe that the complement $S^3\smallsetminus\nu(T^1_{3,3k+2})$ may be obtained by $-\frac{1}{k}$-filling of one component of a link depicted in Figure \ref{fig:hyperbolic-link}. Since this link complement is hyperbolic (as verified by SnapPea, for example, combined with work of Moser \cite{Moser2009}), all but finitely many of the $-\frac{1}{k}$-fillings must be hyperbolic as well.   


\begin{figure}[ht!]\labellist \small
	\pinlabel $-\frac{1}{k}$ at 134 455
	\endlabellist
\includegraphics[scale=0.5]{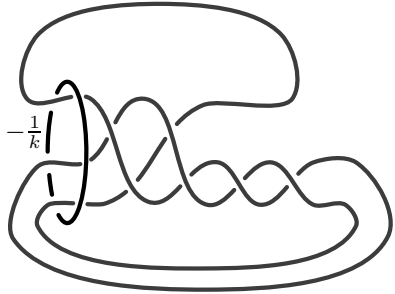}
\caption{The exterior of the knot $T^1_{3,3k+2}$ may be obtained via $-\frac{1}{k}$-surgery on the unknotted component of the hyperbolic link shown.}\label{fig:hyperbolic-link}\end{figure}

It is also easy to verify that the knots $T^1_{3,3k+2}$ are distinct from the pretzel knots given by $T^m_{3,5}$ (considered in the previous section). For example, the Khovanov homology of the former is supported on $k+1$ diagonals, while the latter is supported on exactly $2$ diagonals for every $m$ (see Watson \cite{Watson2010}). In fact, these calculations give an alternate method for proving that the knots $T^1_{3,3k+2}$ are non-torus for $k>0$. We remark that the case $k=0$ gives rise to the torus knot $T^1_{3,2}\simeq T_{2,5}$. 

On the other hand, it is a consequence of the surgery description of this family of knots, together with an unpublished observation due to Hedden, that the knots $T^1_{3,3k+2}$ are L-space knots. Since the genus of these knots is given by $3k+2$, surgery greater than $6k+3$ will always give rise to an L-space.  In this final section, we prove:

\begin{theorem}
\label{thm:twisted}
Suppose that $q$ is a positive integer congruent to $2$ modulo $3$.  If $r > 3q+2$, then $r$-surgery on the twisted torus knot $T_{3,q}^1$ yields a manifold with non-left-orderable fundamental group.
\end{theorem}

Thus, writing $q = 3k+2$ where $k\ge0$, we consider the $(3,3k+2)$ torus knot with one positive twist.  The corresponding knot group is
\[ \pi_q^1 =  \langle a, b | a^2(b^{-k} a) a = b^{2k+1}(b^{-k}a) b^{k+1} \rangle, 
\]
together with the meridian
\[ \mu = b^{k+1} a^{-1} = (b^{-k}a)^{-1}b^{-2k-1}a^2(b^{-k}a)
\]
and surface framing 
\[ s = \mu^{3q+2} \lambda = a^2 (b^{-k}a)a(b^{-k}a) = b^{2k+1}(b^{-k}a)b^{k+1}(b^{-k}a),
\]
where $s=\mu^{3q+2}\lambda$. Theorem \ref{thm:twisted} follows from an application of Corollary \ref{crl:main-criteria} together with the following proposition:

\begin{proposition}
If $>$ is any left-ordering of $\pi_q^1$, then $s>1$ implies $\mu^N s >1 $ for all $N>0$.
\end{proposition}
\begin{proof} We consider all possible choices of signs for the generators $a$ and $b$ of $\pi_q^1$.  There are 6 cases to consider in general, though in the present  setting it is possible to reduce the argument very quickly to the cases where $a$ and $b$ are both positive elements. 

To this end, observe that if both $1>a$ and $1>b$ then \[s=b^{2k+1}(b^{-k}a)b^{k+1}(b^{-k}a)=b^{k+1}aba\] is a product of negative elements, hence $1>s$. Thus at least one of the generators $a$ or $b$ must be positive given the hypothesis that $s>1$.

Similarly, if $b>1>a$ then $b^k>a$ and $1>b^{-k}a$ so that \[s=a^2(b^{-k}a)a(b^{-k}a)\] is a product of the negative elements $a$ and $b^{-k}a$, hence $1>s$. 

On the other hand, if $a>1>b$ then \begin{align*} \mu^Ns = s\mu^N 
&= a^2(b^{-k}a)a(b^{-k}a)(b^{-k}a)^{-1}(b^{-2k-1}a^2)^N(b^{-k}a) \\
&= a^2(b^{-k}a)a(b^{-2k-1}a^2)^N(b^{-k}a) 
\end{align*}
hence $\mu^Ns>1$ as a product of the positive elements $a$ and $b^{-1}$ for all $N\ge0$.  

It remains then to consider the case in which both $a$ and $b$ are positive. Suppose that $s>1$, and let $N>0$ be any positive integer. As before if $\mu>1$, then the conclusion $\mu^N s >1 $ follows trivially, so we suppose that $\mu$ is negative, observe that $\mu^{-1}=ab^{-k-1}>1$, and record
\begin{equation}
\label{eqn:rewrite}
ab^{-k}= \mu^{-1}b >1
\end{equation}
for use in the remainder of the proof. We consider two cases. 

\textbf{Case 1:} $b>a>1$

In this case we have that $b^\ell>a$ for any positive integer $\ell$, hence $a^{-1}b^\ell>1$. Now write \[\mu^N=(b^{k+1}a^{-1})^N=b^{k+1}(a^{-1}b^{k+1})^{N-1}a^{-1}\] and \[s=a(a(b^{-k}a))^2=a(\mu^{-1}ba)^2\] using Equation (\ref{eqn:rewrite}). This gives rise to \begin{align*}
\mu^Ns &= b^{k+1}(a^{-1}b^{k+1})^{N-1}a^{-1}a(\mu^{-1}ba)^2 \\
&= b^{k+1}(a^{-1}b^{k+1})^{N-1}(\mu^{-1}ba)^2
\end{align*}
which we recognize as the product of the positive elements $a$, $b$, $a^{-1}b^{k+1}$ and $\mu^{-1}$, hence $\mu^Ns>1$ for all $N\ge 0$. 

\textbf{Case 2:} $a>b>1$

In this case we note that either $b^{2k+2}a^{-2}>1$ or $a^2b^{-2k-1}>1$ since, should these both be negative elements, the product \[(b^{2k+2}a^{-2})(a^2b^{-2k-1})=b\] would be negative and give a contradiction. This observation gives rise to two subcases. 

{\bf Subcase I:} $w=b^{2k+2}a^{-2}>1$

Write \[\mu^N = (b^{k+1}a^{-1})^N=b^{-k-1}((b^{2k+2}a^{-2})(ab^{-k-1}))^{N}b^{k+1}=b^{-k-1}(w\mu^{-1})^{N}b^{k+1}\] and \[s=a(a(b^{-k}a))^2=a(\mu^{-1}ba)^2=a\mu^{-1}ba\mu^{-1}ba\] (using Equation (\ref{eqn:rewrite}), as before), so that \begin{align*}
 \mu^Ns = s\mu^N 
 &= a\mu^{-1}ba\mu^{-1}bab^{-k-1}(w\mu^{-1})^{N}b^{k+1} \\
 &= a\mu^{-1}ba\mu^{-1}b\mu^{-1}(w\mu^{-1})^{N}b^{k+1}
\end{align*}
so that $\mu^Ns>1$ as a product of the positive elements $a$, $b$, $\mu^{-1}$ and $w$.

{\bf Subcase II:} $w=a^2b^{-2k-1}>1$

Write \[(b^{-2k-1}a^2)^N =  b^{-2k-1}(a^2b^{-2k-1})^{N-1}a^2 = b^{-2k-1}w^{N-1}a^2 \] so that \begin{align*} \mu^Ns =s\mu^N &= 
a^2(b^{-k}a)a(b^{-k}a)(b^{-k}a)^{-1}(b^{-2k-1}a^2)^N(b^{-k}a) \\
&= 
a^2(b^{-k}a)ab^{-2k-1}w^{N-1}a^2(b^{-k}a) \\
&= 
a^2b^{-k}w^{N}a^2b^{-k}a \\
&= a\mu^{-1}bw^{N}a\mu^{-1}ba
\end{align*}
using Equation (\ref{eqn:rewrite}) in the last step. This shows that $\mu^Ns>1$, again as a product of the positive elements $a$, $b$, $\mu^{-1}$ and $w$. 
\end{proof}

\bibliographystyle{plain}

\bibliography{LO}

\end{document}